\documentclass[letterpaper, 11pt,  reqno]{amsart}
 
\usepackage{amsmath,amssymb,amscd,amsthm,amsxtra, esint, bbm, enumerate, relsize, xcolor}
\usepackage{mathtools} 
\usepackage[letterpaper,margin=2.3cm]{geometry} 

\usepackage[implicit=true]{hyperref}
 
\newtheorem{theorem}{Theorem} [section]

\newtheorem{lemma}[theorem]{Lemma}
\newtheorem{proposition}[theorem]{Proposition}

\theoremstyle{definition}
\newtheorem{remark}[theorem]{Remark}

\DeclareMathOperator{\Law}{Law}
 
\newcommand{\1}{\mathbbm 1}

\newcommand{\noi}{\noindent}
\newcommand{\Z}{\mathbb{Z}}
\newcommand{\R}{\mathbb{R}}
\newcommand{\C}{\mathbb{C}}
\newcommand{\T}{\mathbb{T}}

\renewcommand{\H}{\mathcal H}

\let\P= \undefined
\newcommand{\P}{\mathbb{P}}

\newcommand{\E}{\mathbb{E}}

\newcommand{\al}{\alpha}
\newcommand{\be}{\beta}
\newcommand{\dl}{\delta}

\newcommand{\nb}{\nabla}

\newcommand{\vp}{\varphi}

\newcommand{\ep}{\varepsilon}

\newcommand{\g}{\gamma}

\newcommand{\ld}{\lambda}
\newcommand{\Ld}{\Lambda}

\newcommand{\ft}{\widehat}

\newcommand{\wt}{\widetilde}
\newcommand{\cj}{\overline}

\renewcommand{\O}{\Omega}

\newcommand{\les}{\lesssim}
\newcommand{\ges}{\gtrsim}

\newcommand{\jb}[1]
{\langle #1 \rangle}

\newcommand{\ind}{\mathbbm 1}

\newcommand{\N}{\mathbb{N}}

\newcommand{\ph}{\varphi}

\renewcommand{\H}{\mathcal{H}}

\newcommand{\eps}{\ep}
 
\renewcommand{\ft}{\widehat}
 
\numberwithin{equation}{section}
\numberwithin{theorem}{section}
 
\newcommand{\HH}{\mathbb H_a}
\newcommand{\id}{\mathrm{id}}

\let\Re=\undefined\DeclareMathOperator*{\Re}{Re}
\let\Im=\undefined\DeclareMathOperator*{\Im}{Im}

\newcommand{\WW}[1]{e^{\frac1p\int_\T|D_t^{-1}\jb\nb u+#1|^p\,dx}\,\ind_{\{\|D_t^{-1}\jb\nb u+#1\|_{L^2}^2\le K\}}}
\newcommand{\WWz}[1]{e^{\frac1p\int_\T|D_t^{-1}\jb\nb u|^p\,dx+\langle u,D_t\jb\nb #1\rangle}\,\ind_{\{\|D_t^{-1}\jb\nb u\|_{L^2}^2\le K\}}}
\newcommand{\HF}[2]{\mathcal H\big(\tfrac1p{\textstyle\int_\T}|\cdot|^p\big)(#1)[#2]}
\newcommand{\WV}[1]{e^{\frac1p\int_\T|#1+D_t^{-1}\jb\nb u|^p-\frac1p\int_\T|#1|^p-\langle D_t^{-1}\jb\nb u,\,|#1|^{p-2}#1\rangle}\,\ind_{\{\|#1+D_t^{-1}\jb\nb u\|_{L^2}^2\le K\}}}
\newcommand{\RT}{R_v(D_t^{-1}\jb\nb T_{t,v}u)}
\newcommand{\IT}{\ind_{\{\|v+D_t^{-1}\jb\nb T_{t,v}u\|_{L^2}^2\le K\}}}

\title[Log Sobolev inequality for focusing NLS]{A remark on the log-Sobolev inequality for the Gibbs measure of the focusing Schr\"odinger equation}
\author{Guopeng Li, Jiawei Li, Leonardo Tolomeo}

\address{
Guopeng Li, 
School of Mathematics and Statistics, Beijing Institute of Technology, Beijing 100081, China,}
 
\email{guopeng.li@bit.edu.cn}

\address{
Jiawei Li\\
School of Mathematics\\
The University of Edinburgh\\
and The Maxwell Institute for the Mathematical Sciences\\
James Clerk Maxwell Building\\
The King's Buildings\\
Peter Guthrie Tait Road\\
Edinburgh\\ 
EH9 3FD\\
 United Kingdom\\
}

\email{jiawei.li@ed.ac.uk}
 
\address{
Leonardo Tolomeo\\
School of Mathematics\\
The University of Edinburgh\\
and The Maxwell Institute for the Mathematical Sciences\\
James Clerk Maxwell Building\\
The King's Buildings\\
Peter Guthrie Tait Road\\
Edinburgh\\ 
EH9 3FD\\
 United Kingdom\\
}

\email{l.tolomeo@ed.ac.uk}
 
\begin{document}


\begin{abstract}
 
   We consider the question of showing a log-Sobolev inequality for the Gibbs measure of the focusing Schr\"odinger equation built by Lebowitz-Rose-Speer (1988), formally given by 
    $$ d\rho \sim \exp\big(\frac 1 p\int_{\mathbb T} |u|^p d x - \frac 12\int_{\mathbb T} |\nabla u|^2 d x - \frac 12\int_{\mathbb T} |u|^2 d x\big) \mathbbm 1_{\| u \|_{L^2(\mathbb T)}^2 \le K}dud\overline{u}. $$
    When $2 \le p \le 4$, we show that these measures indeed satisfy a log-Sobolev inequality. When $4< p<6$, we show that the multiscale Bakry-\'Emery criterion of Bauerschmidt and Bodineau, with the heat-kernel decomposition of the covariance, cannot be applied to the measure $\rho$: at small scales $t$, the Hessians of the renormalised potentials are not bounded above by any integrable rate. The obstruction comes from fields concentrated at the scale $t^{1/(p-4)}$.
 
\end{abstract}
 
\maketitle
 
\section{Introduction}

In this paper, we consider the focusing Gibbs measures
for the nonlinear Schr\"odinger equations (NLS), 
initiated in the seminal paper by Lebowitz, Rose, and Speer~\cite{Lebowitz1988}.
A Gibbs measure $\rho$
is a probability measure on 
functions\,/\,distributions 
whose density can be \emph{formally} given by
\begin{align}
d\rho = \frac 1 Z e^{- H(u)} d u , 
\label{Gibbs1}
\end{align}
\noindent
where  $H(u)$ is a Hamiltonian functional and $Z$ is a normalisation constant, called the partition function. 
In particular, we investigate the logarithmic
Sobolev inequalities (see \eqref{LSI}) for the focusing 
nonlinear Schr\"odinger equation (NLS) 
on the circle $\T = \R / (2\pi \Z)$:
\begin{align}
i u_t +  \Delta u + |u|^{p-2} u = 0, 
\label{NLS1}
\end{align}
for which the Hamiltonian $H(u)$ is given by 
\begin{align}
H(u) = \frac 12\int_\T |\nabla u|^2 d x - \frac 1 p\int_\T |u|^p d x.
\label{Gibbs2}
\end{align}
 The study of the equation~\eqref{NLS1}
 from the viewpoint of the (non-)equilibrium statistical mechanics 
 has received wide attention; see for example \cite{Lebowitz1988, Bo94, BO96, BO97,Tzv1,  Tzv2, LMW, BBulut, 
CFL, DNY2, Bring3}. See also \cite{BOP4} for a survey on the subject, more from the dynamical point of view.
 
The main difficulty in constructing the focusing Gibbs measures
comes from the high degree $p>2$ of the negative term in the Hamiltonian \eqref{Gibbs2}.
This makes the problem extremely different from the defocusing case, which would correspond to the well-studied $\Phi^p$ model of quantum field theory. In \cite{Lebowitz1988}, Lebowitz, Rose, and Speer suggested considering the Gibbs measure with an extra $L^2$-cutoff 
\begin{align}
\begin{split}
d\rho 
& = \frac{1}{Z_K} e^{- H(u) - \frac12 M(u)} \1_{\{M(u) \le K\}} d u \\
& = \frac{1}{Z_K} \exp\bigg(\frac{1}{p} \int_\T |u|^p dx\bigg) 
\1_{\{M(u) \le K\}} d\mu(u), 
\label{Gibbs3}
\end{split}
\end{align}
where $M(u)$ denotes the mass functional 
\begin{equation}
M(u) = \int_\T |u|^2 dx 
\label{mass}
\end{equation}
and $\mu$ is the Gaussian measure given as in
\eqref{baseGM} below.
This choice is suitable for the study of the statistical mechanics of NLS, since the mass $M(u)$ is conserved by the flow of NLS, and represents the ``generalised number of particles". 
In a series of papers \cite{Lebowitz1988, Bo94, OST, TolWeb}, the authors showed the following.
\begin{proposition}
    Consider the Gaussian measure 
    \begin{equation}\label{baseGM}
        d\mu \sim \exp\Big( - \frac 12 \int_\T |u|^2 dx - \frac 12 \int_\T |\nabla u|^2 dx  \Big)dud\cj u.
    \end{equation}
    Then the measure $\rho$ in \eqref{Gibbs3} is well-defined in the following cases: 
    \begin{enumerate}
        \item $ p < 6$, $K \ge 0$, 
        \item $ p = 6$, and $K \le \|Q \|_{L^2(\R)}^2$, where $Q$ is the ground state of the following elliptic equation on $\R$:
        $$ \Delta Q +  Q^5 - 2 Q = 0. $$
    \end{enumerate}
    Moreover, we have that in all the cases above, the Radon-Nikodym derivative $f = \frac{d \rho}{d \mu}$ satisfies 
    \begin{enumerate}
        \item If $ p < 6$, $K \ge 0$, $f \in L^q(\mu)$ for every $q <\infty $,
        \item If $ p = 6$, and $K \le \|Q \|_{L^2}^2$, then $f \in L^q(\mu)$ for every 
        $$q \le \Big(\frac{\|Q\|_{L^2}^2}{K}\Big)^2, $$
        \item If $p =6$, and $K = \|Q \|_{L^2}^2$, then for every $\al < 1$, $f \in L (\log L)^\al $, 
        see \cite[Proposition 6]{HoNi}. 
    \end{enumerate}
\end{proposition}
 
Here, we focus our attention on the logarithmic Sobolev inequality (LSI) for the Gibbs measures with mass-cutoff. Namely, we ask whether the measure
$\rho$ in \eqref{Gibbs3} satisfies an inequality of the form 
\begin{equation}\label{LSI}
    \int |F(u)|^2 \log \Big( \frac{|F(u)|^2}{\int |F(u)|^2 \, d \rho(u)}\Big) d \rho(u) 
    \;\le\; C \int \| \nabla F(u) \|_{L^2}^2 \, d \rho(u),
\end{equation}
for every suitable smooth functional $F : L^2(\T) \to \R$ and some constant $C>0$.
 
When such an inequality \eqref{LSI} holds, we define
\begin{equation}\label{df-lsi}
    \mathrm{LS}(\rho)
    := \inf\Big\{ C>0 \,:\, \eqref{LSI} \text{ holds for all admissible } F \Big\},
\end{equation}
that is, $\mathrm{LS}(\rho)$ is the best (smallest) constant for which \eqref{LSI} is valid.
If there is no finite constant $C$ for which \eqref{LSI} holds, we set
\begin{equation}\label{df-nlsi}
    \mathrm{LS}(\rho) := \infty.
\end{equation}

Logarithmic Sobolev inequalities \eqref{LSI} are important 
functional inequalities firstly established by Gross in 
\cite{Gross}
for Gaussian measures, where its equivalence to hypercontractivity was also shown. LSI is a powerful tool in showing the uniqueness of invariant measure and exponential rate of convergence of the Markov semigroup towards the invariant measure in models from statistical mechanics and quantum field theory, see e.g. \cite{HS, Z1, Z2}.
See \cite{BGL, GZ} for comprehensive reviews.

In \cite{BaEm}, Bakry and \'Emery provide a nice sufficient condition for LSI under log-concave measures. 
For a Gibbs measure $\rho$ as in \eqref{Gibbs1}, 
Bakry-\'Emery criterion 
(see Proposition~\ref{prop:BaEm}) implies 
LSI holds when the Hamiltonian $H$ is strictly 
convex. For applications of Bakry-\'Emery criterion, 
see \cite{CS,Ledoux, GZ}.


In the context of Gibbs measures for NLS, \cite{BlowerBD} was, to the best of our knowledge, the first work attempting to address a logarithmic Sobolev inequality for the focusing NLS with $2 \le p \le 4$ via the Bakry-\'Emery criterion. We however point out that in \cite{BlowerBD} the authors seem to overlook the complications deriving from the presence of the cutoff function. More specifically, the authors first prove that an appropriate potential is convex on the ball, and then refer to the result in \cite{BaEm} to conclude the log-Sobolev inequality. However, the result in \cite{BaEm} does not apply to this very general setting. Indeed, if instead the cutoff function in \eqref{Gibbs3} is put on a non-convex subset of the ball, then the log-Sobolev inequality \eqref{LSI} might fail (counterexamples can be constructed easily when the set is disconnected and $F$ is constant on the connected components), so a certain level of care around the cutoff is necessary.
 
In our first main result we solve these issues, and show the following.
 
\begin{theorem}\label{thm:LS}
Let $2 \le p \le 4$ and $K>0$, and let $\rho$ be the focusing Gibbs measure
with $L^2$-cutoff at level $K$ given in \eqref{Gibbs3}, associated with the
nonlinear Schr\"odinger equation \eqref{NLS1} on $\T$.
Then $\rho$ satisfies a logarithmic Sobolev inequality of the form
\eqref{LSI}. In particular,
\[
   \mathrm{LS}(\rho) < \infty,
\]

\noi
in the sense of 
definition
\eqref{df-lsi}.

\end{theorem}
 
\begin{remark}
By \cite[Theorem 4.9]{GZ}, and Lemma \ref{NRconvergenceweak}, the log-Sobolev inequality proven in Theorem \ref{thm:LS} implies uniqueness of the invariant measure and spectral gap for an appropriately defined Markov process. This process can be characterised as the limit as $R\to \infty$ of the Markov process obtained by solving the SPDE
$$ \partial_t u + (1 -\Delta) u = |u|^{p-2}u -16 R (\|u\|_{L^2}^2 - K)_+^7 u + 2\xi, $$
where $\xi$ is a complex-valued space-time white noise.
 
In analogy to the finite dimensional setting, we expect that this Markov process can also be given by the solution to the reflected SPDE
$$\partial_t u + (1 -\Delta) u = |u|^{p-2}u - K^{-\frac12} u \dot L_t + 2\xi, $$
where $L_t$ denotes the local time of $u$ at the boundary $\{ \| u \|_{L^2}^2 = K \}$. However, showing the equivalence of the two Markov processes is outside the scope of this paper.
\end{remark}

When $p > 4$, the situation is more complex.
In general, establishing LSI beyond
the strict convexity assumption on
$H$ is challenging, especially on infinite-dimensional spaces.
In recent years, there have been several breakthroughs by Bauerschmidt, Bodineau, and collaborators, see
\cite{BB1, BB2, BBD, BaDa}, that culminated with the proof of uniqueness for the $\ph^4_2$ measure in high temperature in \cite{BDW}. These results have been obtained via the use of a ``multiscale Bakry-\'Emery formula", see \cite[Theorem 2.5]{BB2}, that we now describe in the particular case of the measure $\rho$.
For $t>0$, let $D_t$ be the Fourier multiplier
\[
D_t:=\jb\nb\big(1-e^{-t\jb\nb^2}\big)^{-\frac12},
\qquad\text{so that}\qquad
D_t^2=\frac{\jb\nb^2}{1-e^{-t\jb\nb^2}} .
\]
Define the family of Gaussian measures $\bar \mu_t$ \emph{formally} given by
\begin{align*}
    d\bar \mu_t (u)
    \sim
    \exp\Big(-\frac 12 \jb{u, D_t^2u}\Big) dud\cj u.
\end{align*}
Note that if $u$ is a random variable with $\Law(u) = \mu$ in \eqref{baseGM}, then
\begin{equation}
    \Law\big(D_t^{-1}\jb\nb u\big) = \bar \mu_t.
\end{equation}
We then set
\begin{align*}
    V_t(\vp)
    =  \log\Big(\E_{\bar \mu_t} \Big[ \exp\Big(\frac 1p \int_\T |u+ \ph|^p dx\Big)\1_{ \| u + \ph \|_{L^2}^2 \le K} \Big]\Big)
\end{align*}
for every $\vp\in H^1$
(note that our $V_t$ is the negative of the
renormalized potential defined in \cite{BB2}).
Applied heuristically to the infinite dimensional setting, Theorem~2.5 in \cite{BB2} for the Gibbs measure \eqref{Gibbs3} reads as follows. Suppose there
is $\dot\ld\in L^1_{\rm loc}([0,\infty))$ such that, with $\ld_t:=\int_0^t\dot\ld_s\,ds$,
$$ \int_0^\infty e^{-2 \ld_t} dt < \infty, $$

\noi
and such that the Hessian of $ V_t(\vp)$,
denoted by $\mathcal{H}(V_t)(\vp)$, satisfies
\begin{align}
    \mathcal{H}(V_t)(\vp)[e^{-\frac{t\jb{\nb}^2}{2}}w, e^{-\frac{t\jb{\nb}^2}{2}}w] \le  \frac 12 \| w \|_{H^1}^2
    - \dot \ld_t \| w \|_{L^2}^2
    \label{BB}
\end{align}
for all $t>0$, all $\ph$ and all $w$. Then LSI holds, with
$\mathrm{LS}(\rho)\le2\int_0^\infty e^{-2\ld_t}dt$.

For $t\in(0,\infty)$ and $\ph\in H^1(\T)$ we also consider the probability measures
\begin{equation}\label{rhotphdef}
d\rho^t_\ph(u)=\frac1{Z^t_\ph}\exp\Big(\frac1p\int_\T\big|D_t^{-1}\jb\nb u+\ph\big|^pdx\Big)\1_{\{\|D_t^{-1}\jb\nb u+\ph\|_{L^2}^2\le K\}}\,d\mu(u),
\end{equation}
where $\mu$ is as in \eqref{baseGM} and $Z^t_\ph$ is the normalisation constant. Since
$\Law(D_t^{-1}\jb\nb u)=\bar\mu_t$ when $\Law(u)=\mu$, we have $V_t(\ph)=\log Z^t_\ph$.
 
The second main result of this paper shows that, for $4<p<6$, \eqref{BB} fails at
small $t$.
 
\begin{theorem}\label{thm:noLS}
Let $4<p<6$ and $K>0$. There are $c>0$ and $t_0>0$ such that, for every $t\in(0,t_0]$, there
are $\ph_t\in H^1(\T)$ and a real $w_t\in H^1(\T)$ with $\|w_t\|_{L^2}=1$ such that
\begin{equation}\label{NoHope}
\H V_t(\ph_t)\big[e^{-\frac t2\jb\nb^2}w_t,e^{-\frac t2\jb\nb^2}w_t\big]
\ \ge\ \frac12\|w_t\|_{H^1}^2+\frac ct.
\end{equation}
In particular, no $\dot\ld\in L^1_{\rm loc}([0,\infty))$ satisfies \eqref{BB}.
\end{theorem}
 
The fields $\ph$ achieving \eqref{NoHope} are of the form
\[
\ph=v-D_t^{-2}\big(|v|^{p-2}v\big),
\]
where $v$ is a bump of mass $K/2$ at the scale $t^{1/(p-4)}\ll\sqrt t$. For such $\ph$,
if $u$ is distributed according to $\rho^t_\ph$, then the field
$D_t^{-1}\jb\nb u+\ph$ is close to a Gaussian field centred at $v$
(Proposition~\ref{s4:PROP:crit}). What generates the divergence is the fact that according to $\rho^t_\ph$, for an appropriate choice of $\ph$, we have that the quantity $\mathcal H(\int |u|^p)$ diverges at the correct rate.
 
We emphasize
that this result only shows a structural obstruction to the current
multiscale convexity-based methods.
It does \emph{not} imply the failure of the logarithmic Sobolev inequality itself,
which remains an open question in the super-quartic regime.
 
 \subsection*{Errata} In the previous version of this paper, Theorem \ref{thm:noLS} was stated incorrectly, since it claimed that the RHS of \eqref{NoHope} could be made equal to $\infty$ at $t = \infty$. This was due to a term that was incorrectly dropped in the computation of the quantity $\H V_\infty$. However, the same mechanism is the one that generates the divergence that leads to \eqref{NoHope}, even if the proof is more involved.
 
 \subsection*{Declaration on the use of generative AI}
During the preparation of this work, the authors used Claude Opus 5.5 (Anthropic) to assist with drafting and revising parts of Section~\ref{s4:SEC},
including the proof of Theorem~\ref{thm:noLS}, and with proofreading. The authors reviewed and verified all
mathematical content and take full responsibility for the paper.

\subsection*{Acknowledgements} 
 The authors would like to thank Fabian H\"ofer for pointing out the mistake in the original version of this paper. 

L.T. is funded by the European Union, via the ERC Starting Grant ``CritPDEsRand'', grant agreement no.\ 101219486.
Views and opinions expressed are however those of the authors only and do not necessarily reflect those of the European Union or the European Research Council Executive Agency.
Neither the European Union nor the granting authority can be held responsible for them.
 
\section{Preliminaries} 
\label{SEC:prelim}

Throughout the paper we work with complex-valued functions
$u:\T\to\C$ and view $L^2(\T;\C)=L^2(\T)$ as a real Hilbert space with scalar
product
\[
 \langle f,g\rangle := \Re\int_{\T} f(x)\,\overline{g(x)}\,dx.
\]

\noi
In particular, $\|u\|_{L^2}^2 = \langle u,u\rangle$ and all gradients
and Hessians are taken with respect to this real structure.

We use $C>0$ to denote various constants, which may vary line by line. We also use $A\les B$ to denote an estimate of the form $A\leq CB$ for some constant $C>0$. We denote by $dx$ the Lebesgue measure on the torus, normalised so that $\int_\T dx = 1$.

Given $N \in \N$, 
we denote by  $P_N$
 the projector onto the frequencies 
$\{|n| \leq N\}$
defined by 
\begin{align*}
 \ft{P_N f}(n)= \ind_{\{|n|\leq N\}} \ft f(n).
\end{align*}
We remark that due to the boundedness of the Hilbert transform on the torus, we have that 
\begin{align}\label{PNconvergence}
    \| P_N f - f \|_{L^p} \to 0 
\end{align}
as $N\to \infty$, for every $f \in L^p$, $1<p<\infty$.
 
 


 
Finally, we set $x_+ = \max\{x,0\}$.

To show Theorem~\ref{thm:LS}, 
we will use the following celebrated criterion.
\begin{proposition}[Bakry-\'Emery \cite{BaEm}] \label{prop:BaEm}
Consider a probability measure on $\R^n$ (or a linear sub-space) of the form 
$$ d\nu(x) \sim e^{-H(x)}dx, $$
and assume that there is $\ld >0 $ such that as quadratic forms:
\begin{equation}\label{strongconvexity}
    \H(H(x)) \ge \ld \cdot \mathrm{id}.
\end{equation}

\noi
Then, we have
$$ \mathrm{LS}(\nu) \le \frac{2}{\ld} < \infty.$$
\end{proposition}

\subsection{Gaussian measure and Gibbs measure}
The Gaussian measure $\mu$ we introduced formally in \eqref{baseGM} can be defined
rigorously as the law of 
\begin{align}
    u(x) = \sum_{n\in \Z} \frac{g_n}{\jb{n}}e^{inx},
    \label{baseGMr}
\end{align}
where $\{g_n\}_{n\in \Z}$ is a family of 
mutually independent complex-valued centred Gaussian random variables. 
By standard bounds on Gaussian random variables, we have the following. 
\begin{lemma}
\label{LEM:GaussInt}
   Let $u$ be a function-valued random variable with $\Law(u) = \mu$ in \eqref{baseGM}. Then for every $ q < \infty$, and every $x \in \T$, we have that 
   \begin{equation}
       \E |u(x)|^q = \E \|u \|_{L^q(\T)}^q < \infty.
   \end{equation}
\end{lemma}
Moreover, using \eqref{baseGMr}, 
we can deduce the following fact.
\begin{lemma} \label{0sphere}
    Let $K > 0$, and consider the set 
    $$ S_K = \{ u \in L^2(\T): \| u \|_{L^2}^2 = K\}.$$
    Then 
    $$ \mu(S_K) = 0.$$
\end{lemma}
\begin{proof}
    The proof is completely analogous to \cite[Lemma 2.7]{RoSeToWa}.
\end{proof}
 
We also need exponential integrability of the nonlinearity, uniformly in the frequency
truncation.
 
\begin{lemma}\label{LEM:unifexp}
Let $2\le p<6$, $\be>0$ and $\wt K>0$. Then
\[
\sup_{N\in\N}\E_\mu\Big[\exp\Big(\be\int_\T|P_Nu|^pdx\Big)\1_{\{\|P_Nu\|_{L^2}^2\le\wt K\}}\Big]<\infty .
\]
\end{lemma}
\begin{proof}
By \cite[Proposition 3.1]{TolWeb} (see also \cite{Lebowitz1988, Bo94}),
$\E_\mu\big[\exp\big(\be\int_\T|u|^pdx\big)\1_{\{\|u\|_{L^2}^2\le\wt K+1\}}\big]<\infty$.
Under $\mu$, $P_Nu$ and $(1-P_N)u$ are independent, and the law of $(1-P_N)u$ is
symmetric. Let $A_N:=\{\|(1-P_N)u\|_{L^2}^2\le1\}$. This event depends only on
$(1-P_N)u$ and is symmetric, so $\E_\mu\big[(1-P_N)u\,\big|\,P_Nu,A_N\big]=0$. The map
$z\mapsto\exp(\be\int_\T|z|^p)$ is convex, so by Jensen's inequality
\[
\E_\mu\Big[e^{\be\int_\T|u|^p}\1_{A_N}\,\Big|\,P_Nu\Big]\ge\mu(A_N)\,e^{\be\int_\T|P_Nu|^p}.
\]
On $A_N\cap\{\|P_Nu\|_{L^2}^2\le\wt K\}$ we have $\|u\|_{L^2}^2\le\wt K+1$. Hence
\[
\mu(A_N)\,\E_\mu\Big[e^{\be\int_\T|P_Nu|^p}\1_{\{\|P_Nu\|_{L^2}^2\le\wt K\}}\Big]
\le\E_\mu\Big[e^{\be\int_\T|u|^p}\1_{\{\|u\|_{L^2}^2\le\wt K+1\}}\Big].
\]
Finally, $\mu(A_N)>0$ for every $N$, since $A_N\supset\{\|u\|_{L^2}\le1\}$ and $0$ lies in the support of $\mu$, see \cite[Theorem 3.6.1]{Bogachev},
and $\mu(A_N)\to1$ as $N\to\infty$, since $\E_\mu\|(1-P_N)u\|_{L^2}^2\to0$. So
$\inf_N\mu(A_N)>0$.
\end{proof}

\subsection{Bou\'e-Dupuis variational formula} In order to state the Bou\'e-Dupuis variational formula,
as in \cite{BG,TolWeb, OOT1,OOT2, OST20}, we first need to introduce some notations.
Let $W_t$ be  a cylindrical Brownian motion in $L^2(\T)$
given by
\begin{align*}
W_t = \sum_{n \in \Z} B_n(t) e^{inx},
\end{align*}

\noi
where  
$\{B_n\}_{n \in \Z}$ is a family of mutually independent complex-valued\footnote{By convention, we normalize $B_n$ so that $\Re B_n$ and $\Im B_n$ are independent real Brownian motions with $\E[(\Re B_n(t))^2]=\E[(\Im B_n(t))^2]=t$. Then $W$ has identity covariance with respect to the real scalar product $\jb{\cdot,\cdot}$, and the law of $Y_1$ below is $\mu$ with the normalisation \eqref{s4:cov}.} Brownian motions on a probability space $(\O, \mathcal{F}, \mathbb{P})$. Let $(\mathcal{F}_t)_{t\in [0, 1]}$ be the
$\P$-completion of the filtration generated by 
$(W_t)_{t\in [0,1]}$.
We then define a centered Gaussian process $Y_t$
by 
\begin{align}
Y_t
=  \jb{\nabla}^{-1}W_t.
\label{P2}
\end{align}

\noi
Note that 
we have ${\P}\circ (Y_1)^{-1} = \mu$. By setting  $Y^N = P_N Y $, 
we have   $\P\circ (Y^N_1)^{-1} = (P_N)_*\mu$, 
i.e.~the push-forward of $\mu$ under $P_N$.

Next, let $\HH$ denote the space of drifts, 
containing processes that are progressively measurable 
with respect to $(\mathcal{F}_t)_{t\in [0,1]}$, 
belonging to 
$\dot H^1([0,1]; H^1(\T))$, and which are $0$ in $0$.
We now state the variational formula~\cite{BD, Ust};
in particular, see Theorem 7 in \cite{Ust}.
 
\begin{lemma}[Bou\'e-Dupuis variational formula]\label{LEM:var3}
	Let $Y$ be as in \eqref{P2} and fix $N \in \mathbb{N}$.
	Suppose that  $V:C^\infty(\T) \to \R$
	is measurable, and satisfies $\E\big[|V(P_NY_1)|^p\big] < \infty$
	and $\E\big[e^{p'V(P_NY_1)} \big] < \infty$, for some $1 < p < \infty$, with 
    $$ \frac1p + \frac 1 {p'} = 1.$$
	Then, we have
\begin{align} \label{BDN}
	\log \E\Big[e^{V(P_N Y_1)}\Big]
	= \sup_{\Theta \in \HH}
	\E\bigg[ V(P_N Y_1 + P_N \Theta(1)) - \frac{1}{2} \int_0^1 \| \dot\Theta(t) \|_{H^1_x}^2 dt \bigg],
\end{align}
where the expectation $\E = \E_\P$
is with respect to the underlying probability measure~$\P$. 
\end{lemma}
 
 
\section{Proof of Theorem \ref{thm:LS}}
\label{SEC:LS}
Throughout this section, we consider the measure 
\begin{equation}
    d\rho_{\Ld}(u) = \frac{1}{Z_\Ld} \exp\Big(- \Ld \int_\T |u|^2 dx\Big) d \rho(u), 
    \label{rhoLd}
\end{equation}

\noi
where $Z_\Ld$ is an appropriate normalisation constant for a given parameter $\Ld>0$.
Since $\rho$ is supported on the set $\{ \| u \|_{L^2}^2 \le K \}$, we have that 
\begin{equation} \label{LSequivalence}
   \exp(-\Ld K) \mathrm{LS}(\rho) \le  \mathrm{LS}(\rho_\Ld) \le \exp( \Ld K) \mathrm{LS}(\rho).
\end{equation}

\noi
We have the following lemma.

\begin{lemma} \label{NRconvergenceweak}
    Let $p < 6$, and let $N \in \N$, $R > 0$
    and $\Lambda>0$. Consider the measures 
    \begin{align*}
d  &\rho_{\Ld,N,R} \\
 &=
 \frac{1}{Z_{\Ld,N,R}} \exp\Big( - \Ld \int_\T |P_Nu|^2 dx + \frac 1p \int_\T |P_Nu|^p dx- R(\|P_Nu\|_{L^2}^2 - K)_+^8 \Big) d \mu. 
    \end{align*}
    Then we have that
    \begin{equation*}
        \lim_{N \to \infty} \lim_{R \to \infty} \rho_{\Ld, N, R} = \rho_\Ld,
    \end{equation*}
    where $\rho_\Ld$ is given as
    in \eqref{rhoLd} and the limits are taken in total variation.
\end{lemma}

\begin{proof}
 
 
    It is enough to show that 
    \begin{align*}
    \lim_{N \to \infty}& \lim_{R \to \infty}   \exp\Big( - \Ld \int_\T |P_Nu|^2 dx + \frac 1p \int_\T |P_Nu|^p dx - R(\|P_Nu\|_{L^2}^2 - K)_+^8 \Big)    \\
        &\quad \,= \exp\Big( -\Ld \int_\T |u|^2 dx + \frac 1p \int_\T |u|^p dx\Big) \1_{\| u\|_{L^2}^2
    \le K}
    \end{align*} 
    in $L^1(\mu)$.
    We first check pointwise convergence and then verify a uniform
    integrability bound.

    We have that 
    \begin{align*}
 \lim_{N \to \infty} &  \lim_{R \to \infty} \exp\Big( - \Ld \int_\T |P_Nu|^2 dx + \frac 1p \int_\T |P_Nu|^p dx - R(\|P_Nu\|_{L^2}^2 - K)_+^8 \Big)         \\
        &\quad \,= \exp\Big( -\Ld \int_\T |u|^2 dx + \frac 1p \int_\T |u|^p dx\Big) \1_{\| u\|_{L^2}^2
    \le K}
    \end{align*} 
    pointwise, so it is enough to show that for any $q>1$, 
    \begin{equation}
        \sup_{N \in \N} \limsup_{R \to \infty}  \int \exp\Big(  \frac qp \int_\T |P_Nu|^p - qR(\|P_Nu\|_{L^2}^2 - K)_+^8 \Big) d\mu(u) < \infty. 
    \end{equation}
    By dominated convergence theorem, we have that 
    \begin{align*}
     \sup_{N \in \N}       &   \limsup_{R \to \infty}  \E_\mu \bigg[  \exp\Big(  \frac qp \int_\T |P_Nu|^p dx - qR(\|P_Nu\|_{L^2}^2 - K)_+^8 \Big) \bigg] \\
        &\,\,=  \sup_{N \in \N} \E_\mu \bigg[ \exp\Big(  \frac qp \int_\T |P_Nu|^p dx \Big)\1_{\| P_Nu \|_{L^2}^2 \le K} \bigg],
    \end{align*}

\noi
    which is finite by Lemma~\ref{LEM:unifexp}.
\end{proof}

\begin{proof}[Proof of Theorem \ref{thm:LS}]
    By \eqref{LSequivalence}, 
    it suffices to show that for $\Ld$ big enough, $\mathrm{LS}(\rho_\Ld) <\infty$. Moreover, it is enough to show the inequality \eqref{LSI} for bounded cylindrical functions $F$, in the sense that there exists $N_0 \in \N$ such that $F(u) = F(P_{N_0} (u))$. Finally, in view of Lemma \ref{NRconvergenceweak}, the result follows once we show that 
    $$ \mathrm{LS}((P_N)_* \rho_{\Ld,N,R}) \les 1 $$
    uniformly in $N$ and $R$. We have that, on the image of $P_N$ (which is finite dimensional, and can be identified with $\C^{2N+1}$), the measure $(P_N)_* \rho_{\Ld,N,R}$ is of the form 
    \begin{align*}
    & d(P_N)_* \rho_{\Ld,N,R}(u)
    \\
    & \qquad \sim
    \exp\Big(- \Ld \int_\T |u|^2 dx + \frac 1p \int_\T |u|^p dx - R(\|u\|_{L^2}^2 - K)_+^8 - \frac12 \| u \|_{H^1}^2 \Big) du d\cj u.
    \end{align*}
    Therefore, we can apply Proposition \ref{prop:BaEm} to 
    $$ H(u) = \Ld \int_\T |u|^2 dx - \frac 1p \int_\T |u|^p dx +  R(\|u\|_{L^2}^2 - K)_+^8 + \frac12 \| u \|_{H^1}^2. $$
    We have that 
    \begin{equation}
    \begin{aligned}
        \H H(u) &= 2 \Ld \cdot \id - \mathcal H\big(\tfrac1p{\textstyle\int_\T}|\cdot|^p\big)(u) + 
        16R (\|u\|_{L^2}^2 - K)_+^7 \cdot \id \\
        &\phantom{=\ }+ 224R   (\|u\|_{L^2}^2 - K)_+^6 u \otimes u 
        + (1-\Delta),
    \end{aligned}
    \label{HessH}
    \end{equation}
which, since $\mathcal H\big(\tfrac1p{\textstyle\int_\T}|\cdot|^p\big)(u)\le(p-1)|u|^{p-2}$ as operators (see \eqref{HintUB} below), implies that 
\begin{equation*}
\begin{aligned}
  \mathcal{H} H(u) 
  &\ge 2\Lambda \cdot \mathrm{id}
   - (p-1)|u|^{p-2}
   + 16R(\|u\|_{L^2}^2-K)_+^7 \cdot \mathrm{id}
   + (1-\Delta)
\end{aligned}
\end{equation*}
as operators.

    For $w \in H^1$, since $p -2 \le 2$, by Sobolev embeddings and Young's inequality, we have that 
    (for some constant $C > 0$ that can change from line to line)
\begin{align*}
(p-1) \int_\T |u|^{p-2}|w|^2 dx &\le \int_\T C(1 + |u|^2)|w|^2 dx\\
       &\le C(1 + \|u\|_{L^2}^2) \| w \|_{L^\infty}^2 \\
       &\le C(1 + \|u\|_{L^2}^2) \| w \|_{H^\frac 34}^ 2 \\
       &\le C(1 + \|u\|_{L^2}^2) \| w \|_{L^2}^{\frac12} \|w \|_{H^1}^{\frac 32} \\
       &\le C(1 + \|u\|_{L^2}^2)^4 \|w \|_{L^2}^2 + \frac 12 \|w \|_{H^1}^2.
    \end{align*}
    This shows that 
    $$ (p-1) |u|^{p-2}\le C(1 + \|u\|_{L^2}^2)^4 \cdot \id + (1-\Delta)$$
    as operators. Moreover, there exists a constant $\Ld_* = \Ld_*(K)$ such that 
    $$ (\|u\|_{L^2}^2 - K)_+^7 - C(1 + \|u\|_{L^2}^2)^4 \ge - \Ld_*.  $$
    Therefore, for $R > \frac 1 {16}$, we get that 
    as operators
    \begin{align*}
        \H H(u) &\ge (2 \Lambda - \Ld_*)\cdot \id \ge \id
    \end{align*}
    as long as $\Lambda \ge \frac{\Ld_* + 1}{2}$, which shows that 
    $$ \mathrm{LS}((P_N)_* \rho_{\Ld,N,R}) \le 2.$$
\end{proof}
\section{Proof of Theorem \ref{thm:noLS}}
\label{s4:SEC}
 
In this section we compute the Hessian of the renormalised potentials $V_t$ using
the Cameron--Martin theorem, and then prove Theorem~\ref{thm:noLS}. Throughout this section, $4< p<6$ and $K>0$ are fixed.
 
\subsection{Conventions}
 
As in Section~\ref{SEC:prelim},
$\jb{f,g}=\Re\int_\T f\,\cj g\,dx$, and $dx$ is normalised so that
$\int_\T dx=1$. Functions of $\jb\nb$ act as Fourier multipliers:
\[
F(\jb\nb)f=\sum_{n\in\Z}F(\jb n)\ft f(n)e^{inx},
\qquad \ft f(n)=\int_\T f(x)e^{-inx}dx .
\]
In particular, the operator $D_t$ of the introduction has symbol
$\jb n(1-e^{-t\jb n^2})^{-\frac12}\ge\jb n$, it commutes with $\jb\nb$, and $D_t\jb\nb^{-1}$ is
bounded on every $H^s$, with norm at most $(1-e^{-t})^{-\frac12}$.
 
The Gaussian measure $\mu$ of \eqref{baseGM} is normalised so that
\begin{equation}\label{s4:cov}
\E_\mu\big[\jb{u,f}\jb{u,g}\big]=\jb{f,\jb\nb^{-2}g},
\end{equation}
in particular
\begin{equation}
\E_\mu\big[(\Re\ft u(n))^2\big]=\E_\mu\big[(\Im\ft u(n))^2\big]=\jb n^{-2}.
\end{equation}
This corresponds to $\E[(\Re g_n)^2]=\E[(\Im g_n)^2]=1$ in
\eqref{baseGMr}.
 
For $g\in H^{-1}(\T)$, the pairing $\jb{u,g}$ is a centred Gaussian random
variable with variance $\jb{g,\jb\nb^{-2}g}$. The Cameron--Martin theorem
\cite[Corollary 2.4.3]{Bogachev} states that, for $\psi\in H^1(\T)$ and every
measurable $F:L^2(\T)\to[0,\infty]$,
\begin{equation}\label{s4:CM}
\E_\mu\big[F(u+\psi)\big]=\E_\mu\Big[F(u)\exp\Big(\jb{u,\jb\nb^2\psi}-\frac12\|\psi\|_{H^1}^2\Big)\Big].
\end{equation}

\begin{lemma}\label{s4:LEM:int}
Let $t\in(0,\infty)$ and $\ph\in H^1(\T)$. Then
\begin{equation} \label{positivity}
    \E_\mu\big[\WW{\ph}\big]>0,
\end{equation}
and for every $q<\infty$,
\begin{equation}\label{Lqphiqualitative}
    \E_\mu\Big[\exp\Big(\frac qp\int_\T\big|D_t^{-1}\jb\nb u+\ph\big|^p\,dx\Big)\ind_{\{\|D_t^{-1}\jb\nb u+\ph\|_{L^2}^2\le K\}}\Big]<\infty .
\end{equation}
\end{lemma}
 
\begin{proof}

By Sobolev embeddings, we have that 
$$ \| \ph\|_{L^p}\les \|\ph\|_{H^1}, $$
so in particular 
$$ \int |\ph|^p < \infty.$$
Therefore, from the elementary inequality 
$$|a+b|^p \le 2^{p-1} (|a|^p + |b|^p),$$
we have that 
\begin{align*}
    & \E_\mu\Big[\exp\Big(\frac qp\int_\T\big|D_t^{-1}\jb\nb u+\ph\big|^p\,dx\Big)\ind_{\{\|D_t^{-1}\jb\nb u+\ph\|_{L^2}^2\le K\}}\Big] \\
    &\le e^{\frac {2^{p-1}q}{p}\int_\T |\ph|^p}\E_\mu\Big[\exp\Big(\frac {2^{p-1}q}{p}\int_\T\big|D_t^{-1}\jb\nb u\big|^p\,dx \Big)\ind_{\{\|D_t^{-1}\jb\nb u\|_{L^2}^2\le 2K + 2\|\ph\|_{H^1}^2\}}\Big],
\end{align*}
so \eqref{Lqphiqualitative} follows from (an easy modification of) \cite[Proposition 3.1]{TolWeb}. Therefore, in order to show \eqref{positivity}, it is enough to show that 
$$\mu(\ind_{\{\|D_t^{-1}\jb\nb u+\ph\|_{L^2}^2\le K\}}) > 0. $$
However, this is an immediate consequence of Cameron-Martin's theorem, see \eqref{s4:CM}, for $\psi = D_t\jb\nb^{-1}\ph$, together with standard support theorems of Gaussian measures, see \cite[Theorem 3.6.1]{Bogachev}.
\end{proof}

By Lemma~\ref{s4:LEM:int}, $0<Z^t_\ph<\infty$ in 
\eqref{rhotphdef} , so $\rho^t_\ph$ is well defined and
$V_t(\ph)=\log Z^t_\ph$ is finite, for every $t\in(0,\infty)$ and $\ph\in H^1(\T)$.
 
 
\begin{proposition}\label{s4:PROP:CM}
Let $t\in(0,\infty)$ and $\ph\in H^1(\T)$.
 
\smallskip
\noindent
{\rm (i)} We have
\begin{equation}\label{s4:VCM}
\begin{aligned}
V_t(\ph)&=\log\E_\mu\Big[\WWz{\ph}\Big]\\
&\quad-\frac12\|D_t\ph\|_{L^2}^2 .
\end{aligned}
\end{equation}
 
\smallskip
\noindent
{\rm (ii)} For every $h\in H^1(\T)$, the map $s\mapsto V_t(\ph+sh)$ is $C^\infty$ on $\R$, and
\begin{equation}\label{s4:Hess}
\begin{aligned}
\H V_t(\ph)[h,h]
&=\E_{\rho^t_\ph}\Big[\big\langle u,D_t\jb\nb h\big\rangle^2\Big]
-\Big(\E_{\rho^t_\ph}\Big[\big\langle u,D_t\jb\nb h\big\rangle\Big]\Big)^2\\
&\quad-\E_\mu\Big[\big\langle u,D_t\jb\nb h\big\rangle^2\Big],
\end{aligned}
\end{equation}
where $\E_\mu\big[\jb{u,D_t\jb\nb h}^2\big]=\|D_th\|_{L^2}^2$.
\end{proposition}
 
\begin{proof}
(i) We have
\[
D_t^{-1}\jb\nb u+\ph=D_t^{-1}\jb\nb\big(u+D_t\jb\nb^{-1}\ph\big)
\]
and $\|D_t\jb\nb^{-1}\ph\|_{H^1}^2=\|D_t\ph\|_{L^2}^2$.
Hence (i) is \eqref{s4:CM}
with $\psi=D_t\jb\nb^{-1}\ph$ and
\[
F(v)=\exp\Big(\frac1p\int_\T|D_t^{-1}\jb\nb v|^p\Big)\ind_{\{\|D_t^{-1}\jb\nb v\|^2_{L^2}\le K\}} .
\]
 
(ii) By (i), for $s\in\R$,
\[
V_t(\ph+sh)=\log\E_\mu\Big[\WWz{\ph}\,e^{s\jb{u,D_t\jb\nb h}}\Big]-\frac12\|D_t(\ph+sh)\|_{L^2}^2 .
\]
By \eqref{Lqphiqualitative} with $\ph=0$ and the Cauchy--Schwarz inequality, we can
differentiate under $\E_\mu$ any number of times. Let $\nu$ be the probability measure
on $L^2(\T)$ given by
\[
d\nu(u)\sim\WWz{\ph}\,d\mu(u).
\]
Then
\[
\frac{d^2}{ds^2}\Big|_{s=0}V_t(\ph+sh)
=\E_\nu\big[\jb{u,D_t\jb\nb h}^2\big]-\Big(\E_\nu\big[\jb{u,D_t\jb\nb h}\big]\Big)^2-\|D_th\|_{L^2}^2 .
\]
By \eqref{s4:CM} with $\psi=D_t\jb\nb^{-1}\ph$, $\nu$ is the law of $u+D_t\jb\nb^{-1}\ph$
under $\rho^t_\ph$, and
\[
\jb{u+D_t\jb\nb^{-1}\ph,D_t\jb\nb h}=\jb{u,D_t\jb\nb h}+\jb{\ph,D_t^2h}.
\]
Since a constant shift does not change the variance, and
$\|D_th\|_{L^2}^2=\E_\mu[\jb{u,D_t\jb\nb h}^2]$ by \eqref{s4:cov}, this gives \eqref{s4:Hess}.
\end{proof}

In what follows, we work with the infinite dimensional quantity $\frac1p \int_\T |u|^p$. To make things  rigorous, we should work with Fourier projectors $P_N$ and take the limit as $N\to \infty$, however, we omit these steps. 
 
For $v\in L^\infty(\T)$ and $g\in L^2(\T)$ we use the second variation of the
nonlinearity,
\begin{equation}\label{s4:secvar}
\HF{v}{g,g}:=\frac{d^2}{d\eps^2}\Big|_{\eps=0}\frac1p\int_\T|v+\eps g|^pdx
=\int_\T\Big(|v|^{p-2}|g|^2+(p-2)|v|^{p-4}\big(\Re\,\cj vg\big)^2\Big)dx ,
\end{equation}
and we write $\mathfrak h_v g:=|v|^{p-2}g+(p-2)|v|^{p-4}v\,\Re(\cj vg)$. This is the
(real-linear, self-adjoint, non-negative) operator such that
$\jb{g,\mathfrak h_vg}=\HF{v}{g,g}$. Since $0\le(\Re\,\cj vg)^2\le|v|^2|g|^2$, \eqref{s4:secvar} gives
\begin{equation} \label{HintUB}
    \int_\T|v|^{p-2}|g|^2dx\le\HF{v}{g,g}\le(p-1)\int_\T|v|^{p-2}|g|^2dx .
\end{equation}

By Taylor's formula we have, for every $\dl\in L^p$,
\begin{equation}\label{s4:Taylor}
\begin{aligned}
&\frac1p\int_\T|v+\dl|^p-\frac1p\int_\T|v|^p-\big\langle\dl,|v|^{p-2}v\big\rangle
=\frac12\HF{v}{\dl,\dl}+R_v(\dl),\\
&|R_v(\dl)|\le C_p\int_\T\big(|v|^{p-3}|\dl|^3+|\dl|^p\big)dx .
\end{aligned}
\end{equation}
Here we used that the third derivative of $z\mapsto\frac1p|z|^p$ on $\C\simeq\R^2$
is bounded by $c_p|z|^{p-3}$, and $p\ge3$.

\begin{lemma}\label{s4:LEM:recentre}
Let $t\in(0,\infty)$ and $v\in H^1(\T)\cap L^\infty(\T)$, and set
\begin{equation}\label{s4:phiv}
\ph:=v-D_t^{-2}\big(|v|^{p-2}v\big)\in H^1(\T).
\end{equation}
Then the image of $\rho^t_\ph$ under the translation
$u\mapsto u-D_t^{-1}\jb\nb^{-1}\big(|v|^{p-2}v\big)$ is the probability measure
\begin{equation}\label{s4:rhotilde}
d\tilde\rho^{\,t}_v(u):=\frac1{\tilde Z^t_v}\,\WV{v}\,d\mu(u),
\end{equation}
where $\tilde Z^t_v$ is the normalisation constant. Consequently, for every
trigonometric polynomial $h$,
\begin{equation}\label{s4:Hessv}
\H V_t(\ph)[h,h]
=\E_{\tilde\rho^{\,t}_v}\big[\jb{u,D_t\jb\nb h}^2\big]-\Big(\E_{\tilde\rho^{\,t}_v}\big[\jb{u,D_t\jb\nb h}\big]\Big)^2-\|D_th\|_{L^2}^2 .
\end{equation}
By \eqref{s4:Taylor}, the exponent in \eqref{s4:rhotilde} equals
$\frac12\HF{v}{D_t^{-1}\jb\nb u,D_t^{-1}\jb\nb u}+R_v(D_t^{-1}\jb\nb u)$.
\end{lemma}
 
\begin{proof}
\eqref{s4:rhotilde} is an easy consequence of Cameron--Martin's theorem and formula \eqref{s4:CM}.  
Moreover, since
\[
\big\langle u-D_t^{-1}\jb\nb^{-1}(|v|^{p-2}v),D_t\jb\nb h\big\rangle=\jb{u,D_t\jb\nb h}-\big\langle|v|^{p-2}v,h\big\rangle
\]
differs from $\jb{u,D_t\jb\nb h}$ by a constant, this does not change the variance. Therefore, 
$$ \E_{\tilde\rho^{\,t}_v}\big[\jb{u,D_t\jb\nb h}^2\big]-\Big(\E_{\tilde\rho^{\,t}_v}\big[\jb{u,D_t\jb\nb h}\big]\Big)^2 = \E_{\rho^t_\ph}\Big[\big\langle u,D_t\jb\nb h\big\rangle^2\Big]
-\Big(\E_{\rho^t_\ph}\Big[\big\langle u,D_t\jb\nb h\big\rangle\Big]\Big)^2,$$
from which we get \eqref{s4:Hessv}.
\end{proof}
 In what follows, we want to write the measure 
 $$ d\tilde\rho^{\,t}_v(u) \sim \exp\Big(\frac12\HF{v}{D_t^{-1}\jb\nb u,D_t^{-1}\jb\nb u}+R_v(D_t^{-1}\jb\nb u)\Big) d \mu(u)$$
as a Gaussian measure with inverse covariance given by the quadratic form
$$ \jb{u,\jb\nb^2u}-\HF{v}{D_t^{-1}\jb\nb u,D_t^{-1}\jb\nb u},$$
multiplied by a ``good" density. In the following Lemma, we study basic properties of the Gaussian part of this measure.

\begin{lemma}
\label{s4:LEM:gauss}
Let $t>0$ and $v\in L^\infty(\T)$, and let
\begin{equation}\label{s4:theta}
\theta_t(v):=\sup_{f\neq0}\frac{\HF{v}{D_t^{-1}f,D_t^{-1}f}}{\|f\|_{L^2}^2}
\end{equation}
be the norm of the non-negative self-adjoint operator $D_t^{-1}\mathfrak h_vD_t^{-1}$ on $L^2$. Then
\begin{equation}\label{s4:thetabound}
\theta_t(v)\le(p-1)\sum_{n\in\Z}\frac{1-e^{-t\jb n^2}}{\jb n^2}\int_\T|v|^{p-2}dx .
\end{equation}
Assume moreover that $\theta:=\theta_t(v)<1$, and let
\[
T_{t,v}:=\jb\nb^{-1}\big(1-D_t^{-1}\mathfrak h_vD_t^{-1}\big)^{-\frac12}\jb\nb .
\]
Then the following hold.
\begin{enumerate}
\item[(i)] For every measurable
$F\ge0$,
\[
\frac{\E_\mu\big[F(u)\,e^{\frac12\HF{v}{D_t^{-1}\jb\nb u,D_t^{-1}\jb\nb u}}\big]}{\E_\mu\big[e^{\frac12\HF{v}{D_t^{-1}\jb\nb u,D_t^{-1}\jb\nb u}}\big]}
=\E_\mu\big[F(T_{t,v}u)\big].
\]
\item[(ii)] $\|T_{t,v}g\|_{H^1}^2\le(1-\theta)^{-1}\|g\|_{H^1}^2$ for every $g$.
\item[(iii)] $\|D_t^{-1}\jb\nb T_{t,v}g\|_{H^1}^2\le(1-\theta)^{-1}\|g\|_{H^1}^2$ for every $g$.
\item[(iv)] For every $x\in\T$, $\displaystyle\E_\mu\big|D_t^{-1}\jb\nb T_{t,v}u(x)\big|^2\le\frac2{1-\theta}\sum_{n\in\Z}\frac{1-e^{-t\jb n^2}}{\jb n^2}$.
\item[(v)] For every trigonometric polynomial $h$,
\[
\E_\mu\big[\jb{T_{t,v}u,D_t\jb\nb h}^2\big]-\|D_th\|_{L^2}^2\ \ge\ \HF{v}{h,h},
\qquad
\E_\mu\big[\jb{T_{t,v}u,D_t\jb\nb h}^2\big]\le\frac{\|D_th\|_{L^2}^2}{1-\theta}.
\]
\end{enumerate}
\end{lemma}
 
\begin{proof}
For every $x\in\T$ and $f\in L^2$, the Cauchy--Schwarz inequality gives
\[
|D_t^{-1}f(x)|^2\le\sum_{n\in\Z}\frac{1-e^{-t\jb n^2}}{\jb n^2}\,\|f\|_{L^2}^2 .
\]
Together with \eqref{HintUB}, this gives \eqref{s4:thetabound}. If
$\theta<1$, then $0\le D_t^{-1}\mathfrak h_vD_t^{-1}\le\theta<1$ on $L^2$, so $T_{t,v}$ is well
defined. By \eqref{s4:cov}, $T_{t,v}u$ is a centred Gaussian field with covariance
\[
T_{t,v}\jb\nb^{-2}T_{t,v}^*=\jb\nb^{-1}\big(1-D_t^{-1}\mathfrak h_vD_t^{-1}\big)^{-1}\jb\nb^{-1}.
\]
 
(i) Up to a finite dimensional approximation, $\mu$ has density proportional to
$e^{-\frac12\jb{u,\jb\nb^2u}}$, and
\[
\HF{v}{D_t^{-1}\jb\nb u,D_t^{-1}\jb\nb u}=\jb{\jb\nb u,D_t^{-1}\mathfrak h_vD_t^{-1}\jb\nb u}.
\]
So the probability measure on the left-hand side of (i) has density proportional to
\[
e^{-\frac12\jb{\jb\nb u,(1-D_t^{-1}\mathfrak h_vD_t^{-1})\jb\nb u}}.
\]
This is the centred Gaussian measure with the covariance displayed above, that is,
the law of $T_{t,v}u$.
 
(ii) We have $\|T_{t,v}g\|_{H^1}=\|(1-D_t^{-1}\mathfrak h_vD_t^{-1})^{-\frac12}\jb\nb g\|_{L^2}\le(1-\theta)^{-\frac12}\|g\|_{H^1}$.
 
(iii) The symbol of $D_t^{-1}\jb\nb$ is $(1-e^{-t\jb n^2})^{\frac12}\le1$, so $D_t^{-1}\jb\nb$ is a
contraction on $H^1$, and (iii) follows from (ii).

(iv) We have $\Re D_t^{-1}\jb\nb T_{t,v}u(x)=\jb{T_{t,v}u,D_t^{-1}\jb\nb\dl_x}$. By the formula for
the covariance, its variance is
\[
\jb{D_t^{-1}\dl_x,(1-D_t^{-1}\mathfrak h_vD_t^{-1})^{-1}D_t^{-1}\dl_x}\le\frac{\|D_t^{-1}\dl_x\|_{L^2}^2}{1-\theta}=\frac1{1-\theta}\sum_{n\in\Z}\frac{1-e^{-t\jb n^2}}{\jb n^2}.
\]
The same bound holds for the imaginary part, with $i\dl_x$ in place of $\dl_x$.
 
(v) By the formula for the covariance,
\[
\E_\mu\big[\jb{T_{t,v}u,D_t\jb\nb h}^2\big]=\big\langle D_th,\big(1-D_t^{-1}\mathfrak h_vD_t^{-1}\big)^{-1}D_th\big\rangle .
\]
For a self-adjoint operator $A$ with $0\le A\le\theta<1$ we have
$(1-A)^{-1}=1+A+A(1-A)^{-1}A\ge1+A$ and $(1-A)^{-1}\le(1-\theta)^{-1}$. We apply this
with $A=D_t^{-1}\mathfrak h_vD_t^{-1}$. Since
$\jb{D_th,D_t^{-1}\mathfrak h_vD_t^{-1}D_th}=\jb{h,\mathfrak h_vh}=\HF{v}{h,h}$, both claims follow.
\end{proof}
 
\begin{lemma}[Exponential moments of the remainder]\label{s4:LEM:BD}
Let $t>0$, $q\ge1$ and $\theta_0\in(0,1)$, and let $v\in L^\infty(\T)$ with
\[
\|v\|_{L^2}^2\le K,\qquad (p-1)\sum_{n\in\Z}\frac{1-e^{-t\jb n^2}}{\jb n^2}\int_\T|v|^{p-2}dx\le\theta_0 .
\]
Then
\[
\log\E_\mu\Big[e^{q|R_v(D_t^{-1}\jb\nb T_{t,v}u)|}\ind_{\{\|v+D_t^{-1}\jb\nb T_{t,v}u\|_{L^2}^2\le K\}}\Big]\le C,
\]
where $C<\infty$ depends only on $p$, $q$, $\theta_0$ and $K$.
\end{lemma}

\begin{proof}
We follow the argument of \cite[Proposition~3.1]{TolWeb}. By \eqref{s4:thetabound},
$\theta_t(v)\le\theta_0$, so $T_{t,v}$ is well defined. Below, $C$ and the implicit constants in $\les$ depend only on $p$, $q$,
$\theta_0$ and $K$. Let
\[
F(u):=q|R_v(D_t^{-1}\jb\nb T_{t,v}u)|\ind_{\{\|v+D_t^{-1}\jb\nb T_{t,v}u\|_{L^2}^2\le K\}},
\qquad F_{M,L}:=(-L)\vee(M\wedge F).
\]
The quantity to be bounded is at most $\log\E_\mu[e^{F}]$. By dominated convergence as
$L\to\infty$, and then monotone convergence as $M\to\infty$, it suffices to bound
$\log\E_\mu[e^{F_{M,L}}]$ uniformly in $M$ and $L$. Since $F_{M,L}$ is bounded, we can apply
Lemma~\ref{LEM:var3}. With $Y$ as in \eqref{P2}, using
$\int_0^1\|\dot\Theta\|_{H^1}^2\ge\|\Theta(1)\|_{H^1}^2$ and $F_{M,L}\le F_+$, we get
\[
\log\E_\mu\big[e^{F_{M,L}}\big]\le\E\Big[\sup_{\Theta\in H^1}\Big(F_+(Y_1+\Theta)-\frac12\|\Theta\|_{H^1}^2\Big)\Big].
\]

Fix $\Theta\in H^1$. If $\|v+D_t^{-1}\jb\nb T_{t,v}(Y_1+\Theta)\|_{L^2}^2>K$, then $F_+(Y_1+\Theta)=0$. Otherwise,
by \eqref{s4:Taylor}, Young's inequality with exponents $\frac{p-2}{p-3}$ and $p-2$, and
\eqref{HintUB}, for every $\eta\in(0,1]$ we have
\begin{align}
F_+(Y_1+\Theta)&\le qC_p\int_\T\big(|v|^{p-3}|D_t^{-1}\jb\nb T_{t,v}(Y_1+\Theta)|^3+|D_t^{-1}\jb\nb T_{t,v}(Y_1+\Theta)|^p\big)\notag\\
&\le\eta\,\HF{v}{D_t^{-1}\jb\nb T_{t,v}(Y_1+\Theta),D_t^{-1}\jb\nb T_{t,v}(Y_1+\Theta)}\notag\\
&\quad+C_\eta\int_\T|D_t^{-1}\jb\nb T_{t,v}(Y_1+\Theta)|^p\notag\\
&\le2\eta\,\HF{v}{D_t^{-1}\jb\nb T_{t,v}Y_1,D_t^{-1}\jb\nb T_{t,v}Y_1}\tag{I}\label{s4:BD1}\\
&\quad+2\eta\,\HF{v}{D_t^{-1}\jb\nb T_{t,v}\Theta,D_t^{-1}\jb\nb T_{t,v}\Theta}\tag{II}\label{s4:BD2}\\
&\quad+C_\eta\int_\T|D_t^{-1}\jb\nb T_{t,v}Y_1|^p\tag{III}\label{s4:BD3}\\
&\quad+C_\eta\int_\T|D_t^{-1}\jb\nb T_{t,v}\Theta|^p ,\tag{IV}\label{s4:BD4}
\end{align}
where $C_\eta$ depends only on $p$, $q$ and $\eta$. By \eqref{s4:theta} and
Lemma~\ref{s4:LEM:gauss}(ii),
\[
\HF{v}{D_t^{-1}\jb\nb T_{t,v}\Theta,D_t^{-1}\jb\nb T_{t,v}\Theta}\le\theta_0\,\|\jb\nb T_{t,v}\Theta\|_{L^2}^2\le\frac{\theta_0}{1-\theta_0}\|\Theta\|_{H^1}^2 ,
\]
so, choosing $\eta$ small enough depending on $\theta_0$ such that $\frac{2\eta \theta_0}{1-\theta_0} \le \frac14$,
the term \eqref{s4:BD2}
is bounded by $\frac{1}{4}\|\Theta\|_{H^1}^2$.
Moreover, since $\|v+D_t^{-1}\jb\nb T_{t,v}(Y_1+\Theta)\|_{L^2}^2\le K$,
\[
\|D_t^{-1}\jb\nb T_{t,v}\Theta\|_{L^2}\le\|v+D_t^{-1}\jb\nb T_{t,v}(Y_1+\Theta)\|_{L^2}+\|v\|_{L^2}+\|D_t^{-1}\jb\nb T_{t,v}Y_1\|_{L^2}\les1+\|D_t^{-1}\jb\nb T_{t,v}Y_1\|_{L^2},
\]
and $\|D_t^{-1}\jb\nb T_{t,v}\Theta\|_{H^1}^2\le(1-\theta_0)^{-1}\|\Theta\|_{H^1}^2$ by Lemma~\ref{s4:LEM:gauss}(iii). By the
Gagliardo--Nirenberg inequality $\int|f|^p\les\|f\|_{L^2}^{\frac p2+1}\|f\|_{H^1}^{\frac p2-1}$ and
Young's inequality with exponents $\frac4{p-2}$ and $\frac4{6-p}$,
\begin{align}
\eqref{s4:BD4}&\le\frac14\|\Theta\|_{H^1}^2\notag\\
&\quad+C\big(1+\|D_t^{-1}\jb\nb T_{t,v}Y_1\|_{L^2}\big)^{\frac{2p+4}{6-p}} .\tag{V}\label{s4:BD5}
\end{align}
Hence, for every $\Theta\in H^1$,
\[
F_+(Y_1+\Theta)-\frac12\|\Theta\|_{H^1}^2\le\eqref{s4:BD1}+\eqref{s4:BD3}+\eqref{s4:BD5},
\]
and the right-hand side does not depend on $\Theta$.

Finally, $D_t^{-1}\jb\nb T_{t,v}Y_1$ has the law of $D_t^{-1}\jb\nb T_{t,v}u$ under $\mu$. By
Lemma~\ref{s4:LEM:gauss}(iv), for every $x\in\T$, $D_t^{-1}\jb\nb T_{t,v}Y_1(x)$ is a centred Gaussian vector
in $\R^2$ with
\[
\E|D_t^{-1}\jb\nb T_{t,v}Y_1(x)|^2\les\sum_{n\in\Z}\frac{1-e^{-t\jb n^2}}{\jb n^2}\les1.
\]
Hence the expectations of \eqref{s4:BD3} and \eqref{s4:BD5} are $\les1$, and, by \eqref{HintUB}
and the second assumption,
\[
\begin{aligned}
\E\,\eqref{s4:BD1}&\le2\eta(p-1)\int_\T|v|^{p-2}\,\sup_{x\in\T}\E|D_t^{-1}\jb\nb T_{t,v}Y_1(x)|^2\\
&\les(p-1)\sum_{n\in\Z}\frac{1-e^{-t\jb n^2}}{\jb n^2}\int_\T|v|^{p-2}\le\theta_0 .
\end{aligned}
\]
This proves the claim.
\end{proof}

\begin{proposition}
\label{s4:PROP:crit}
Let $\theta_0\in(0,1)$. There exists $\gamma_0\ll1$,
depending only on $p$, $K$ and $\theta_0$, such that the following holds; the implicit constant in \eqref{s4:crit} below also depends only on $p$, $K$ and $\theta_0$. Let $t\in(0,1]$ and
$v\in H^1\cap L^\infty(\T)$ satisfy
\[
\|v\|_{L^2}^2\le\frac K2,\qquad (p-1)\sum_{n\in\Z}\frac{1-e^{-t\jb n^2}}{\jb n^2}\int_\T|v|^{p-2}\le\theta_0,
\qquad
\gamma:=t^{\frac34}\int_\T|v|^{p-3}+t^{\frac14}\le\gamma_0 .
\]
Let $\ph$ be given by \eqref{s4:phiv}. Then, for every trigonometric polynomial $h$,
\begin{equation}\label{s4:crit}
\Big|\H V_t(\ph)[h,h]-\Big(\E_\mu\big[\jb{T_{t,v}u,D_t\jb\nb h}^2\big]-\|D_th\|_{L^2}^2\Big)\Big|
\les \gamma\,\|D_th\|_{L^2}^2 .
\end{equation}
\end{proposition}

\begin{proof}
By \eqref{s4:thetabound}, $\theta_t(v)\le\theta_0$. By Lemma~\ref{s4:LEM:recentre}, \eqref{s4:Taylor} and
Lemma~\ref{s4:LEM:gauss}(i),
\begin{equation}\label{s4:Hessratio}
\begin{aligned}
\H V_t(\ph)[h,h]+\|D_th\|_{L^2}^2
&=\frac{\E_\mu\big[\jb{T_{t,v}u,D_t\jb\nb h}^2\,e^{\RT}\IT\big]}{\E_\mu\big[e^{\RT}\IT\big]}\\
&\quad-\bigg(\frac{\E_\mu\big[\jb{T_{t,v}u,D_t\jb\nb h}\,e^{\RT}\IT\big]}{\E_\mu\big[e^{\RT}\IT\big]}\bigg)^2 .
\end{aligned}
\end{equation}
Below, $C$ and the implicit constants in $\les$ depend only on $p$, $K$ and $\theta_0$.

First of all, note that $\g_0 \ll 1$ implies $t \ll 1$.
Since $1-e^{-s}\sim\min(s,1)$ for $s\ge0$, we have
$\sum_{n\in\Z}\frac{1-e^{-t\jb n^2}}{\jb n^2}\sim\sum_{n\in\Z}\min(t,\jb n^{-2})\sim t^{\frac12}$.
So Lemma~\ref{s4:LEM:gauss}(iv) shows that, for every $x\in\T$, $D_t^{-1}\jb\nb T_{t,v}u(x)$ is a centred Gaussian
vector in $\R^2$ with $\E_\mu|D_t^{-1}\jb\nb T_{t,v}u(x)|^2\les t^\frac12$. By \eqref{s4:Taylor} and Minkowski's
inequality, since $t\le1$,
\[
\|\RT\|_{L^4(\mu)}\les t^{\frac34}\int_\T|v|^{p-3}+t^{\frac p4}.
\]
Since $\|v\|_{L^2}^2\le\frac K2$, Chebyshev's inequality gives
\[
\mu\big(\|v+D_t^{-1}\jb\nb T_{t,v}u\|_{L^2}^2>K\big)\les\E_\mu\|D_t^{-1}\jb\nb T_{t,v}u\|_{L^2}^2\les t^\frac12 .
\]

\noi
 By Lemma~\ref{s4:LEM:BD} with $q=4$,
\[
\E_\mu\big[e^{4|\RT|}\IT\big]\les1 .
\]
We write
\begin{align*}
&\mu\big(\|v+D_t^{-1}\jb\nb T_{t,v}u\|_{L^2}^2\le K\big) \\
&= \E_\mu\big[e^{\frac45\RT} e^{-\frac45\RT}\IT\big].
\end{align*}
Therefore, by H\"older we have that
\begin{align*}
\E_\mu\big[e^{\RT}\IT\big]& \gtrsim \mu\big(\|v+D_t^{-1}\jb\nb T_{t,v}u\|_{L^2}^2\le K\big)^\frac54\\
&\gtrsim 1 - \frac 54 \mu\big(\|v+D_t^{-1}\jb\nb T_{t,v}u\|_{L^2}^2>K\big) \gtrsim 1
\end{align*}
for $t \ll 1$.

Now let
\[
\eps:=\bigg\|\frac{e^{\RT}\IT}{\E_\mu\big[e^{\RT}\IT\big]}-1\bigg\|_{L^2(\mu)} .
\]
Since the denominator is bounded below,  the elementary inequality $|e^x-1|\le|x|(e^x+1)$ and H\"older's inequality give
\[
\begin{aligned}
\eps&\les\big\|e^{\RT}\IT-1\big\|_{L^2(\mu)}\\
&\le\|\RT\|_{L^4(\mu)}\Big(\E_\mu\big[e^{4\RT}\IT\big]^{\frac14}+1\Big)\\
&\quad+\mu\big(\|v+D_t^{-1}\jb\nb T_{t,v}u\|_{L^2}^2>K\big)^{\frac12}\les\gamma .
\end{aligned}
\]
So $\eps\les\gamma$, in particular, $\eps\ll 1$.

Write  $Y = \jb{T_{t,v}u,D_t\jb\nb h},$ and note that this is a centred Gaussian under $\mu$. We have that, for every density $\rho$ with $\E[\rho]=1,$
\begin{align*}
    &\big|\E[Y^2 \rho] - \E[Y\rho]^2 - \E[Y^2]\big| \\
    &= \big|\E[Y^2 (\rho-1)] - \E[Y(\rho-1)]^2\big| \\
    &\le \E[Y^4]^\frac12 \|\rho-1\|_{L^2} + \E[Y^2] \|\rho-1\|_{L^2}^2 \\
    &\les \E[Y^2]\|\rho-1\|_{L^2}(1 + \|\rho-1\|_{L^2}),
\end{align*}
since $\E[Y^4]=3\E[Y^2]^2$.
Applying this to
$$ \rho = \frac{e^{\RT}\IT}{\E_\mu\big[e^{\RT}\IT\big]},$$
and noting that the RHS of \eqref{s4:Hessratio} is exactly $\E[Y^2 \rho] - \E[Y\rho]^2$,
this gives
\[
\Big|\H V_t(\ph)[h,h]+\|D_th\|_{L^2}^2-\E_\mu\big[\jb{T_{t,v}u,D_t\jb\nb h}^2\big]\Big|
\les\eps\,\E_\mu\big[\jb{T_{t,v}u,D_t\jb\nb h}^2\big].
\]
By Lemma~\ref{s4:LEM:gauss}(v), $\E_\mu\big[\jb{T_{t,v}u,D_t\jb\nb h}^2\big]\le(1-\theta_0)^{-1}\|D_th\|_{L^2}^2$,
 and \eqref{s4:crit} follows, since $\eps\les\gamma$.
\end{proof}

\subsection{Proof of Theorem~\ref{thm:noLS}}
\label{s4:SUB:proof}

\begin{lemma}[Test direction]\label{s4:LEM:dir}
Let $\kappa\in(0,1]$ and $t\in(0,\kappa^2]$. Let $N_t:=\lfloor\kappa t^{-\frac12}\rfloor$ and
\[
w_t(x):=(2N_t+1)^{-\frac12}\sum_{|n|\le N_t}e^{inx},
\qquad\text{so that}\qquad
\ft w_t\sim\ind_{\{|n|\le N_t\}}.
\]
Then $w_t$ is a real trigonometric polynomial, $\|w_t\|_{L^2}=1$, and
$\frac12\|w_t\|_{H^1}^2\le\kappa^2t^{-1}$.
Moreover, for every real $v$ supported in $\{|x|\le\sqrt t\}$,
\begin{equation}\label{s4:gain}
\HF{v}{e^{-\frac t2\jb\nb^2}w_t,e^{-\frac t2\jb\nb^2}w_t}\ges\kappa\,t^{-\frac12}\int_\T|v|^{p-2}dx ,
\end{equation}
where the implicit constant depends only on $p$.
\end{lemma}

\begin{proof}
We have $w_t=(2N_t+1)^{-\frac12}\big(1+2\sum_{n=1}^{N_t}\cos(nx)\big)$, which is real, and, by
Parseval's identity, $\|w_t\|_{L^2}=1$ and
\[
\frac12\|w_t\|_{H^1}^2=\frac1{2(2N_t+1)}\sum_{|n|\le N_t}\jb n^2=\frac12+\frac{N_t(N_t+1)}6\le\frac{\kappa^2}t,
\]
since $1\le N_t\le\kappa t^{-\frac12}$. Let
$g:=e^{-\frac t2\jb\nb^2}w_t=(2N_t+1)^{-\frac12}\sum_{|n|\le N_t}e^{-\frac t2\jb n^2}\cos(nx)$.
For $|n|\le N_t$ we have $t\jb n^2\le t+\kappa^2\le2$ and, for $|x|\le\sqrt t$, $|nx|\le\kappa\le1$, so that
$\cos(nx)\ge\frac12$. Hence, for $|x|\le\sqrt t$,
\[
g(x)\ge\frac{(2N_t+1)^{\frac12}}{2e}\ges\kappa^{\frac12}t^{-\frac14}.
\]
For real $v$ and real $g$, $\HF{v}{g,g}=(p-1)\int|v|^{p-2}g^2$, and \eqref{s4:gain} follows.
\end{proof}

\begin{proof}[Proof of Theorem~\ref{thm:noLS}]
Fix once and for all a real, even $\chi\in C_c^\infty((-1,1))$ with
$0\le\chi\le1$ and $\chi\not\equiv0$. In this proof, the implicit constants in $\les$ depend
only on $p$, $K$ and $\chi$. For $t\in(0,1]$, let
$v_t(x):=\al\chi(x/\ell)$ for $|x|\le\pi$, where $\al>0$ and $\ell\in(0,\pi]$ are chosen so that
\begin{equation}\label{s4:choice}
\|v_t\|_{L^2}^2=\frac K2,\qquad (p-1)\sum_{n\in\Z}\frac{1-e^{-t\jb n^2}}{\jb n^2}\int_\T|v_t|^{p-2}dx=\frac12 .
\end{equation}
Since $\int_\T|v_t|^rdx=\al^r\ell\,\frac1{2\pi}\int_\R\chi^r\,dy$ and $p>4$, this determines
$\al$ and $\ell$ uniquely for $t$ small. By \eqref{s4:choice}, $\al^2\ell$ is a constant depending only on $K$ and $\chi$, and,
since $\sum_{n\in\Z}\frac{1-e^{-t\jb n^2}}{\jb n^2}\sim t^{\frac12}$, we have $\al^{p-2}\ell\sim t^{-\frac12}$. Hence
\[
t^{-\frac12}\sim\frac{\al^{p-2}\ell}{\al^2\ell}=\al^{p-4},\qquad\ell\sim\al^{-2}\sim t^{\frac1{p-4}}\ll t^\frac12.
\]
So $v_t$ is smooth, real and supported in $\{|x|\le\ell\}$, and
$\ell^2/t\sim t^{\frac2{p-4}-1}\to0$, because $p<6$. Moreover,
\[
t^{\frac34}\int_\T|v_t|^{p-3}+t^{\frac14}\les t^{\frac34}\al^{p-3}\ell+t^{\frac14}\les t^{\frac14}\ \longrightarrow\ 0 .
\]
By \eqref{s4:choice}, $\int_\T|v_t|^{p-2}\ges t^{-\frac12}$. Hence, by Lemma~\ref{s4:LEM:dir}, there is
$c_1>0$, which does not depend on $\kappa$, such that, for every $\kappa\in(0,1]$ and every $t\le\kappa^2$ small enough,
\begin{equation}\label{s4:limitgain}
\HF{v_t}{e^{-\frac t2\jb\nb^2}w_t,e^{-\frac t2\jb\nb^2}w_t}\ge\frac{c_1\kappa}t ,
\end{equation}
where $w_t$ and $N_t$ are as in Lemma~\ref{s4:LEM:dir}.
Let $\ph_t$ be given by \eqref{s4:phiv} with $v=v_t$, and let $h:=e^{-\frac t2\jb\nb^2}w_t$, which is a trigonometric polynomial.
For $t$ small, $v_t$ satisfies the assumptions of Proposition~\ref{s4:PROP:crit} with $\theta_0=\frac12$, and $\g\les t^{\frac14}$.
Moreover, the symbol $\jb n^2(e^{t\jb n^2}-1)^{-1}$ of $D_t^2e^{-t\jb\nb^2}$ is at most $t^{-1}$, so that
$\|D_th\|_{L^2}^2\le t^{-1}\|w_t\|_{L^2}^2=t^{-1}$. Therefore, denoting by $C$ the implicit constant in \eqref{s4:crit},
by \eqref{s4:crit}, Lemma~\ref{s4:LEM:gauss}(v), \eqref{s4:limitgain} and Lemma~\ref{s4:LEM:dir} we have
\begin{align*}
\H V_t(\ph_t)[h,h]-\frac12\|w_t\|_{H^1}^2
&\ge\E_\mu\big[\jb{T_{t,v_t}u,D_t\jb\nb h}^2\big]-\|D_th\|_{L^2}^2-C\g\,\|D_th\|_{L^2}^2-\frac12\|w_t\|_{H^1}^2\\
&\ge\HF{v_t}{h,h}-\frac{C\g}t-\frac{\kappa^2}t
\ \ge\ \frac{c_1\kappa-C\g-\kappa^2}t .
\end{align*}
Therefore, by taking first $\kappa\ll1$, and then $t\ll1$ so that $\g\ll\kappa$, we deduce that
\begin{align*}
    \H V_t(\ph_t)\big[e^{-\frac t2\jb\nb^2}w_t,e^{-\frac t2\jb\nb^2}w_t\big]
- \frac12\|w_t\|_{H^1}^2 \gtrsim \kappa t^{-1}.
\end{align*}
Since $\|w_t\|_{L^2}=1$, this is \eqref{NoHope}. Finally, if $\dot\ld\in L^1_{\rm loc}([0,\infty))$ satisfied \eqref{BB},
testing \eqref{BB} with $\ph_t$ and $w_t$ would give $\dot\ld_t\le-\frac ct$ for every $t\in(0,t_0]$, so that
$\int_0^{t_0}\dot\ld_s\,ds=-\infty$, a contradiction.
\end{proof}

\end{document}